\documentclass[10pt]{article}
\usepackage{graphicx}
\usepackage{latexsym}
\usepackage{amsmath}
\usepackage{amssymb}
\usepackage{amsthm}
\usepackage{hyperref}
\def \Q {{\mathbb Q}}

\def \Z {{\mathbb Z}}
\def \R {{\mathbb R}}
\def \F {{\mathbb F}}

\def \L {{\mathbb L}}
\def \C {{\mathbb C}}
\def \K {{\mathbb K}}
\def \M {{\mathbb M}}
\def \A {{\mathcal A}}

\def \Bl {{\mathcal B}}
\def \Cl {{\mathcal C}}

\def \J  {{\mathbf J}}
\def \o {{\mathbf 0}}

\def \x {{\mathbf x}}

\def \e {{\mathbf e}}
\def \al {{\alpha}}

\def \la {{\lambda}}
\def \de {{\delta}}

\def \ze {{\zeta}}

\newtheorem{theorem}{Theorem}[section]
\newtheorem{cor}[theorem]{Corollary}
\newtheorem{lemma}[theorem]{Lemma}

\newtheorem{definition}[theorem]{Definition}
\newtheorem{rem}[theorem]{Remark}
\newtheorem{fact}[theorem]{Fact}

\title{Representation of Cyclotomic Fields and Their Subfields}

\author{A. Satyanarayana Reddy\footnote{Department of Mathematics and Statistics, Indian Institute of Technology, Kanpur, India 208016; (e-mail:
satya@iitk.ac.in).} \and Shashank K Mehta \footnote{Department of Computer Science and Engineering,Indian Institute of
Technology, Kanpur, India 208016; (e-mail: skmehta@cse.iitk.ac.in). This work was partly supported by Research-I
Foundation, IIT-Kanpur, Indian Institute of Technology, Kanpur,
 India 208016.}  \and A. K. Lal\footnote{Department of Mathematics and Statistics, Indian Institute of
Technology, Kanpur, India 208016; (e-mail: arlal@iitk.ac.in).  }}

\date{}

\begin{document}
\maketitle

\begin{abstract}

Let $\K$ be a finite extension of a characteristic zero field $\F$. We say
that a pair of $n\times n$ matrices $(A,B)$ over $\F$ represents $\K$ if
$\K \cong \F[A]/\langle B \rangle$, where $\F[A]$ denotes the
subalgebra of $\M_n(\F)$ containing $A$ and $\langle B \rangle$ is an ideal in
$\F[A]$, generated by $B$. In particular,  $A$ is said to represent
 the field $\K$ if  there exists an irreducible polynomial $q(x)\in \F[x]$ which
divides the minimal polynomial of $A$ and $\K \cong \F[A]/\langle q(A) \rangle$.

In this paper, we identify the smallest order
circulant matrix representation for any subfield of a cyclotomic field.
Furthermore, if $p$ is a prime and $\K$ is a subfield of the $p$-th cyclotomic
field, then we obtain a zero-one circulant matrix
$A$ of size $p\times p$ such that $(A,\J)$ represents $\K$, where $\J$ is
the matrix with all entries $1$.
In case, the integer $n$ has at most two distinct prime factors,
we find the smallest order $0,1$-companion matrix that  represents the $n$-th
cyclotomic field. We also find bounds on the size of such
companion matrices when $n$ has more than two prime factors.
\end{abstract}



\section{Introduction and Preliminaries} \label{sec:intro}
In this paper, we will be interested in fields $\F$ that have characteristic $0$.
Thus, one can assume that  $\Q\subseteq \F\subseteq \C$, where $\Q$ is the field of rational numbers and $\C$ is the field of complex numbers.  An element $\al \in \C$ is said to be algebraic over  $\F$, if $\al$ is a
root of a polynomial $f(x) \in \F[x]$. The polynomial $f(x)$ is
said to be the minimal polynomial of $\al$ over $\F$, if $\al$ is a root
of $f(x)$,  $f(x)$ is monic  and  is irreducible in
$\F[x]$. In this paper,  $\M_n(\F)$ will denote the set of all $n \times n$ matrix over $\F$.  All  vector symbols will denote column vectors and they will be written in bold face. Also, the vector of all $1$'s will be denoted by $\e$ and a square matrix with all entries $1$, will be denoted by $\J$. Then $\J = \e
\e^t$, where $\e^t$ denotes the transpose of $\e$. The symbol $\o$ will denote either a vector or a matrix having all entries zero.

Recall that for any $A \in \M_n(\F)$, a celebrated result, commonly known as the Cayley-Hamilton Theorem, states that the matrix $A$ satisfies its own characteristic polynomial. That is, if $\varsigma_A(x) = \det(x I - A)$ is the characteristic polynomial of $A$, then $\varsigma_A(A)$ as an element of $\F[A]$, equals $\o$. Let $S=\{f(x) \in \F[x] | f(A)=\o\}$. Then  $S$ is  an ideal in $\F[x]$ and $S = \langle p(x) \rangle$ for some monic polynomial $p(x) \in \F[x]$.  By definition, $p(x)$ divides $\varsigma_A(x)$ and for any $B \in \F[A]$, there exists a unique
polynomial $g(x) \in \F[x]$, with $\deg(g(x)) < \deg (p(x))$ such that $B = g(A).$ The polynomial $p(x) $  is called the minimal polynomial of $A$ and is  denoted by $p_A(x)$.

We are now ready to state a few results from matrix theory and
abstract algebra.  For proofs and notations related with these results
the reader is advised to refer to the book Abstract Algebra by Dummit $\&$
Foote~\cite{D:F} and Linear Algebra by Hoffman $\&$ Kunze~\cite{H:K}.

\begin{lemma}[Hoffman $\&$ Kunze, Pages~$204, 231$ \cite{H:K}] \label{lem:aa3}
Let $A$ be a square matrix.
\begin{enumerate}
\item  \label{lem:aa3:1} Then A
 is diagonalizable if and only if its minimal polynomial is separable.
\item  \label{lem:aa3:2} Let A be a matrix with distinct eigenvalues. Then a
matrix $B$ commutes
with $A$ if and only if $B$ is a polynomial in $A$.
\end{enumerate}
\end{lemma}

Before stating the next result,  recall that for a monic polynomial $f(x) = x^n - c_{n-1}x^{n-1} - c_{n-2}x^{n-2} - \cdots - c_1
x - c_0 \in \F[x]$, its  companion matrix, denoted
$\Cl(f)$, is defined as
$$ \Cl(f) =
    \begin{bmatrix}
    0 & 1 & 0 & \dots & 0 & 0 \\
    0 & 0 & 1 & \dots & 0 & 0 \\
        \vdots &\vdots & \ddots & \ddots & \vdots & \vdots \\
     0 & 0 & 0 &\dots & 0 & 1\\
     c_0 & c_1 & c_2 & \dots & c_{n-2} & c_{n-1} \\
     \end{bmatrix}.$$
For example, the $n \times n$ matrix $$W_n=
\begin{bmatrix}
    0 & 1 & 0 & \dots & 0 & 0 \\
    0 & 0 & 1 & \dots & 0 & 0 \\
        \vdots &\vdots & \ddots & \ddots & \vdots & \vdots \\
     0 & 0 & 0 &\dots & 0 & 1\\
     1 & 0 & 0 & \dots & 0 & 0 \\
     \end{bmatrix}$$ is the companion matrix of the polynomial $x^n -1.$ Note that $W_n$ is a $0,1$-circulant matrix and $x^n - 1$ is its minimal polynomial. It is well known (for example, see Davis~\cite{davis}) that every
circulant matrix is a polynomial in $W_n$. Due to the above property, the matrix $W_n$ is called the fundamental circulant matrix. The next result also appears in~\cite{H:K}.

\begin{lemma}[Hoffman $\&$ Kunze, Page~$230$ \cite{H:K}] \label{lem:la}
Let $\Cl(f)$ be the companion matrix of $f(x) =  x^n - c_{n-1}x^{n-1} - c_{n-2}x^{n-2} - \cdots - c_1
x - c_0 \in \F[x]$ . Then \begin{enumerate}
\item \label{lem:la:1} $f(x)$ is both the minimal and the
characteristic polynomial of $\Cl(f)$.
\item \label{lem:la:2} all eigenvalues of the companion matrix
$\Cl(f)$ are distinct if and only if $\Cl(f)$ is diagonalizable.
\end{enumerate}
\end{lemma}

The next result is also well known. The proof  can be easily obtained by using basic results in abstract algebra and it also appears in~\cite{D:F}.

\begin{theorem}\label{thm:a3}
Let $p_A(x)$ be  the minimal polynomial of $A\in \M_n(\F)$.
\begin{enumerate}
\item \label{thm:a3:2} Let $g(x) \in \F[x]$ and let $h(x) =
\gcd(g(x),p_A(x))$. Then $\langle g(A) \rangle =\langle h(A) \rangle.$
\item \label{thm:a3:1}
If $q(x)$ is a non-constant factor of $p_A(x)$ in $\F[x]$ then
$\F[A]/\langle q(A)\rangle \cong \F[x]/\langle q(x)\rangle$. In
particular, if $q(x)$ is irreducible and $q(\al) = 0$ for some $\al \in
\C$ then $\F[A]/\langle q(A)\rangle \cong \F[x]/\langle q(x) \rangle
\cong  \F(\al).$ That is, $\F[A]/\langle q(A) \rangle$ is a field.
\end{enumerate}
\end{theorem}

As a corollary of Theorem~\ref{thm:a3}, one has the following result. To state the result, recall that a pair of $n\times n$ matrices $(A,B)$ over $\F$ is said to represent an extension field $\K$ if $\K \cong
\F[A]/\langle B \rangle$, where $\langle B \rangle$ is an ideal in  $\F[A]$
generated by $B$.

\begin{cor}\label{c1} Let $\al \in \C$. Then the matrix pair $(A,B)$ represents $\F(\al)$, a field, if and only if $\al$ is an eigenvalue of $A$, $q(x)$ is the minimal polynomial of $\al$ over $\F$ and
$\langle B\rangle =\langle q(A) \rangle$ in $\F[A]$.
\end{cor}

That is, suppose that $q(x)$ is the minimal polynomial of an eigenvalue $\al$ of $A$. Then the matrix pair $(A,B)$ represents an extension
$\K=\F(\al)$ of $\F$ if and only if  $\langle B \rangle = \langle
q(A)\rangle$. Hence with an abuse of the language, we may say that the matrix
$A$ represent $\K$ to mean that $\K=\F(\al)$, whenever $\al$ is an eigenvalue of
$A$. Also, it is well known that the choice of  $\alpha$ is not unique.
Therefore, depending on the
choice of $\alpha$ the corresponding matrices that represent $\K$ can
vary. This issue becomes significant when we search for a smallest representation (in terms of order).

We are now ready to explain the motivation for our study.
Let $G$ be a finite group and let $n\in \Z^+$. A matrix representation of $G$ is a homomorphism from $G$ into $GL_n(\F)$, where $GL_n(\F)$ is the group of
invertible $n\times n$ matrices with entries from $\F$.
The representation is called faithful if the image of the homomorphism is
isomorphic to $G$.  A similar question arises whether an extension field of $\F$ has a representation in $\M_n(\F)$. For example, let  $\al$ be an algebraic number over $\F$ with $q(x) \in \F[x] $ as  its minimal polynomial.
Then, using Lemma~\ref{lem:la} and Corollary~\ref{c1}, we see that $\F(\al) \cong \F[\Cl(q))]$, where $\Cl(q)$ is the companion matrix of $q(x)$. This leads
to the following natural questions:
\begin{enumerate}
\item does there
exist a matrix $A$ in $\M_n(\F)$  with some  specified properties such that
$\F[A]\cong \F(\al)$?
\item if it exists, what is the smallest possible positive integer $n$?
\end{enumerate}

For example, fix a positive integer $n$ and consider $\ze_n$, a primitive $n$-th
root of unity. Then the polynomial $\Phi_n(x)$ over $\Q$, called the $n$-th
cyclotomic polynomial,  is the minimal polynomial of $\ze_n$ and hence is irreducible over $\Q$. In this case, is
it possible to find a matrix $A$ which is either circulant over $\Q$ or is a
$0,1$-companion matrix of $\Phi_n(x)$ such that  $\Q[A]\cong \Q(\ze_n)$?  It
can easily be checked that this is true only
when $n=1\;$ or $n=2$. For  $n> 2$, such a result is not true. To understand this, recall that  $x^n-1=\prod\limits_{d|n}\Phi_d(x)$, where $\Phi_d(x) \in \Z[x]$ is the minimal polynomial of $\ze_d$ in $\Q[x]$ and for any two integers $s,t$, the notation $s | t$ means that $s$ divides $t$.
Consequently, from Corollary~\ref{c1}, it follows that for each divisor $d$
of $n$,
\begin{equation}\label{eqn:zetad}
\Q[W_n] / \langle \Phi_d(W_n) \rangle \cong \Q(\ze_d).
\end{equation}  In particular,
$\Q[W_n] / \langle \Phi_n(W_n) \rangle \cong \Q(\ze_n).$ That is, in this case, the pair $(W_n,\Phi_n(W_n))$ represents the field
$\Q(\ze_n)$.

To proceed further, we need the following definitions and notations.
A directed graph (in short, digraph) is an ordered pair $X=(V,E)$ that
consists of two sets
$V$, the vertex set, and $E$, the edge set, where $V$ is non-empty and  $E
\subset V \times V$. If $e=(u, v) \in E$ with $u \ne v$ then the edge $e$ is
said to be incident from  $u$ to $v$ or $u$ is said to be the
initial vertex and $v$ the terminal vertex of $e$. An edge $e = (u,u)$ is
called a loop. A digraph is called a  graph if $(u,v) \in E$ whenever $(v,u) \in E$, for any two elements $u,v \in V$.  If $u, v \in V$, then an edge between $u$ and $v$ in the graph $X$ is denoted by $e=\{u,v\}$ and in this case, we say that $e$ is incident with $u$ and $v$ or the vertex $u$ is adjacent to the vertex $v$, or vice-versa.   For any finite set $S$, let $|S|$ denote the number of elements in $S$. Then a graph/digraph is said to be finite, if  $|V| $ (called the
order of $X$)  is finite. A graph is called simple if it has no loops.

Let $X=(V,E)$ be a graph. Then the degree of a vertex $v\in V$, denoted
$d(v)$, is the number of edges incident with it. In case $v$ is  a vertex of a
digraph $X$, one defines in-degree of $v$, denoted $d^+(v)$, as the number
of edges that have $v$ as a terminal vertex and out-degree of $v$, denoted
$d^-(v)$, as the number of edges that have $v$ as an initial vertex.  The simple
graph $X$ that has an edge for each pair of vertices is called a complete
graph, denoted $K_n$, where $n$ is the number of vertices of $X$. A graph with no edge is called a null graph. The  cycle graph on $n$
vertices, say $u_1, u_2, \ldots, u_n$, denoted $C_n$,
 is a simple graph in which  $\{u_i,u_j\}$ is an edge
if and only if $i-j\equiv \pm 1 \pmod n$. A graph on $n$ vertices, say $u_1, u_2,
\ldots, u_n$, denoted $X_n$,  is called a path graph,
if for each $i, 1
\le i \le n-1$, the set $\{ u_i, u_{i+1}\}$ is an edge.
A graph (digraph) $X$ is said to be $k$- regular if
$d(v) = k$ ($d^+(v) = d^-(v) = k$) for all $v \in V$. Unless specified otherwise, all the graphs in this paper are assumed to be finite and simple.

Let $X=(V,E)$ be a digraph. Then  the adjacency matrix of $X$, denoted $A = [a_{ij}]$ is a square matrix of order $|V|$ with $a_{ij}=1$ whenever $(i,j) \in E$ and $0$, otherwise. In case $X$ is a graph then it can be easily seen that $A$ is a symmetric matrix.

Now, let $A$ be the adjacency matrix of a connected $k$-regular graph $X$ on
$n$ vertices. Then, it is well known that $k$ is a simple eigenvalue of
$A$. Thus, the minimal polynomial of $A$ is of the form $(x-k) q(x)$ for
some $q(x) \in \Z[x]$. Note that $k$ is a simple eigenvalue of $A$ implies
that $q(k) \ne 0$ and for any other eigenvalue $\al$ of $A, \; q(\al) =
0$. Then with $q(x)$ as defined, we state the following well known result. We
present the proof for the sake of completeness.

\begin{lemma}[Hoffman~\cite{hof}] \label{lem:regular:J}
Let $X$ be a connected $k$-regular graph on $n$ vertices with
minimal polynomial $(x-k) q(x) \in \Z[x]$. Then the matrix $\J$ equals $ \dfrac{n}{q(k)} q(A)$.
\end{lemma}

\begin{proof}
As $X$ is a $k$-regular graph, its adjacency matrix $A$  satisfies $A \e = k \e$
and hence
\begin{equation} \label{eqn:J} \J A = A
\J = k \J \; {\mbox{ and }} \; q(A) \e = q(k) \e. \end{equation}
Also, the eigenvectors of $A$ can be chosen to form an orthonormal basis
$\mathcal B$ of $\R^n$. Hence $\dfrac{1}{\sqrt{n}} \e \in {\mathcal B}$ and
thus,
for any vector $\x \in {\mathcal B}, \x \ne \e$, $\x^t \e = 0$. Therefore,
$\J \x = \o$ and using Equation~(\ref{eqn:J}), $\J \dfrac{1}{\sqrt{n}} \e =
\dfrac{n}{\sqrt{n}} \e = \bigl(\dfrac{n}{q(k)} q(k) \bigr)
\dfrac{1}{\sqrt{n}} \e= \dfrac{n}{q(k)} q(A) \dfrac{1}{\sqrt{n}}
\e.$
Also, $q(\la) = 0$ for any eigenvalue $\la \ne k$ of $A$ implies that $q(A) \x =
q(\al) \x = \o$. That is, $\dfrac{n}{q(k)} q(A) \x = \o$.

Thus, the image of two matrices $\J$ and
$\dfrac{n}{q(k)} q(A)$ on a basis of $\R^n$ are same and hence the
two matrices are equal. Therefore  $\J = \dfrac{n}{q(k)} q(A).$
\end{proof}

The next corollary is an immediate consequence of
Theorem~\ref{thm:a3} and
Lemma~\ref{lem:regular:J}.

\begin{cor}\label{cor:regular}
Let $A$ be the adjacency matrix of a connected k-regular graph $X$
on $n$ vertices.  Then $\F[A]/\langle \J \rangle \cong \F[x]/\langle
q(x)\rangle$.
\end{cor}

\begin{proof}
Since $q(x)$ is a factor of the minimal polynomial $(x-k)q(x)$
of $A$, using Theorem~\ref{thm:a3} and Lemma~\ref{lem:regular:J}, one has
$\F[x]/\langle q(x)\rangle \cong \F[A]/\langle q(A) \rangle \cong \F[A] /
\left \langle \dfrac{n}{q(k)}q(A) \right \rangle = \F[A] / \langle
\J \rangle.
$
\end{proof}

A.J. Hoffman $\&$ M.H. McAndrew~\cite{hoff} extended Lemma~\ref{lem:regular:J} to digraphs and is stated below. Note that in their paper, regular digraph were referred as strongly regular digraph.

\begin{lemma}[Hoffman $\&$McAndrew(1965)~\cite{hoff}]\label{lem:hoff}
Let $A$ be the adjacency matrix of a digraph $X$. Then there exists a polynomial $g(x) \in \Q[x]$ such that $\J = g(A)$ if and only if $X$ is strongly connected and  regular.
\end{lemma}

In this paper, we determine the smallest order circulant matrix representation for all
subfields of cyclotomic fields.
We also determine smallest order $0,1$-companion matrix representation for cyclotomic
fields in a restricted sense. We only consider those $0,1$-companion matrices
which have $\zeta_n$ as an eigenvalue.
We begin with a review of some facts related  with the representation of cyclotomic fields and their
subfields by circulant matrices in Section~\ref{sec:represent:cyclotomic}.
The results about representations of $p$-th cyclotomic field, $p$ a prime, by
$0,1$-circulant matrices are given in Section~\ref{sec:prime}. In section
\ref{sec7} we present results on the size of the smallest order $0,1$-companion matrix representations of cyclotomic fields.

\section{Representation of cyclotomic fields and their subfields}
\label{sec:represent:cyclotomic}

We start this section with a result about the irreducible factors of the
minimal polynomial of a
companion matrix. The proof of this  result can be easily obtained using the
theory of minimal polynomials and Lemmas~\ref{lem:aa3} and
\ref{lem:la}. Hence we omit the proof.

\begin{lemma} \label{lem:a9}
Let $f(x) \in \F[x]$ be a  monic separable polynomial with irreducible
factors $q_1(x), q_2(x), \ldots, q_k(x)$ in $\F[x]$.  Suppose   $A \in
\M_n(\F)$ commutes with the companion matrix $\Cl(f)$.
\begin{enumerate}\item  Then $A=g(\Cl(f))$ for some $g(x)\in \F[x]$.
\item Let $\al_i$ be a root of $q_i(x)$ and let $\chi_{q_i,g}(x)$ be the minimal
polynomial of $g(\al_i)$ over $\F$. Then the minimal polynomial of $A$ is
the maximal square-free factor of $\prod\limits_{i=1}^k \chi_{q_i,g}(x)$.

In particular, the number of irreducible factors of the minimal polynomial of
$g(A)$ in $\F[x]$ are at most the number of irreducible factors of $f(x)$
in $\F[x]$.
\end{enumerate}
\end{lemma}

Recall that for a fixed positive integer $n$,
$\;\deg(\Phi_n(x)) = \varphi(n)$, where $\varphi(n)$ denotes the
well known Euler-totient function.  The function $\varphi(n)$ also gives the  number of integers between $1$ and $n$ that are coprime to $n$.
We omit the proof of the next result as it directly follows from Lemma~\ref{lem:a9}, Theorem \ref{thm:a3}.\ref{thm:a3:1} and the fact that $x^n-1=\sum\limits_{d|n}\Phi_d(x)$.

\begin{theorem}\label{thm:a10}
Fix a positive integer $n$ and let $A=g(W_n )$, for some  $g(x)\in \Q[x]$ with $1\leq \deg(g(x))\leq n-1$.  Also, for each divisor $d$ of $n$, let
$\chi_{\Phi_d,g}(x)$ be the minimal polynomial of $g(\ze_d)$ over $\Q$. Then
\begin{enumerate}
\item \label{thm:a10:1}
$p_A(x)$, the minimal polynomial of $A$,  is the maximal square free factor of
$\prod\limits_{d|n}\chi_{\Phi_d,g}(x)$ and
$\deg(\chi_{\Phi_d,g}(x))$ divides $\deg(\Phi_d(x))$.
\item \label{thm:a10:2} the number of irreducible factors of $p_A(x)$
is at most the number of divisors of $n$.
\item \label{thm:a10:3}  $\Q\bigl(g(\zeta_d)\bigr) \cong \Q[A]/ \langle
\chi_{\Phi_d,g}(A) \rangle.$
\end{enumerate}
Furthermore, if $n$ is a prime,
say $p$,  then the number of irreducible factors of $p_A(x)$ is exactly two.
One of the factors is of degree $1$ and the degree of the other factor is a
divisor of $\varphi(p)=p-1$.
\end{theorem}

\begin{proof}
Proofs of Part~\ref{thm:a10:1},~\ref{thm:a10:2} and~\ref{thm:a10:3} are direct consequence of  Lemma~\ref{lem:a9} and Theorem~\ref{thm:a3}. For the last statement, note that $n$ is prime and hence it has exactly two factors, namely $1$ and $n$. Hence,  using Lemma~\ref{lem:a9}, it is sufficient to prove that $p_A(x)$ has at least two irreducible factors.

As $A = g(W_n)$, the eigenvalues of $A$ are $g(\ze_n^i)$ for $0 \le i \le n-1$. Now observe that for $i=0$, $g(\ze_n^0) = g(1)\in \Q$ as $g(x) \in \Q[x]$. Therefore, $(x-g(1))$ is an irreducible factor of $p_A(x)$. Also, for some $i, 1 \le i \le n-1$, if $g(\ze_n^i)\neq g(1)$  then the minimal polynomial of $g(\ze_n^i)$ is  another irreducible factor of $p_A(x)$. Hence, $p_A(x)$  has at least two irreducible factors.

Thus, we need to show that $g(\ze_n^i)= g(1)$ cannot hold true for
all $i\in\{1,2\ldots, n-1\}$. On the contrary, assume that
$g(\ze_n^i)= g(1)$ for all $i, 1 \le i \le n-1$. Define, $h(x) =
g(x) - g(1) \in \Q[x]$. Then $h(x)$ has $n$ distinct zeros,
$\ze_n^i, \; 0 \le i \le n-1$. This contradicts the definition of
$h(x)$ and the assumption that $\deg(g(x)) \le n-1$ as a polynomial
$f(x)$ has at most $\deg(f(x))$ zeros over $\C$.
\end{proof}

Theorem~\ref{thm:a10} establishes that, any field which is represented by an
$n\times n$ circulant matrix is a subfield of the $d$-th cyclotomic field
for some $d$ that divides $n$. It also describes
the correspondence between the set of all circulant matrices and the set
of all subfields of cyclotomic fields. It can also be concluded that the minimal
polynomial of every circulant matrix other than the scalar matrix
has at least two irreducible factors. One also concludes the next result and hence the proof is omitted.

\begin{cor}\label{cor:5.2}
Let $A$ be an $n\times n$ circulant matrix. Then $A$ represents a field $\F$ over $\Q$ if and only if $\F$ is a subfield of $\Q[\zeta_d]$, for some $d$ dividing $n$.
\end{cor}

The next result gives the smallest positive integer $d$ for which a field $\L$ over $\Q$ (as a subfield of $\Q[\ze_d]$) is represented by a circulant matrix.

\begin{cor}\label{cor:5.3}
Let $\L$ be a finite extension of $\Q$. If $d$ is the smallest positive integer such that $\L$ is a subfield of $\mathbb{Q}(\zeta_d)$ then the smallest circulant matrix representation of $\L$ is of order $d$.
\end{cor}

\begin{proof}
Since $\L$ is a subfield of $\mathbb{Q}(\zeta_d)$ there exists $g(x)\in \Q[x]$ such that $\L=\Q\bigl(g(\ze_d)\bigr)$.
Thus, by Theorem~\ref{thm:a10}.\ref{thm:a10:3},  $\L$ is
represented by
$A=g(W_d)$.

Now, assume that there exists a $d'\times d'$ circulant matrix
$B= h(W_{d'})$ that represents $\L$, for some $d'<d$. Let
$\chi_{\Phi_{d',h}}(x) \in \Q[x]$ be the minimal polynomial of $h(\ze_{d'})$. Then, using Theorem~\ref{thm:a10}.\ref{thm:a10:3}, one has $\L\cong
\Q[B]/\langle \chi_{\Phi_{d',h}}(B) \rangle \cong \Q[h(\zeta_{d'})]$.
That is,  $\L$ is a subfield of $\Q[\zeta_{d'}]$ as well. This contradicts
 our hypothesis that $d$ was the smallest positive integer such that
$\L$ was a subfield of $\mathbb{Q}[\zeta_d]$. Hence, one has the required
result.
\end{proof}

Let $n,m $ and $a$ be positive integers with $n = 2^a m$ and $a \ge 1$.
Then it is known that $\Q[\zeta_{m}] \cong \Q[\ze_n]$, whenever $m$ is odd and $a=1$.
Let $n=2^a\cdot m$ for some odd positive integer $m$. Then, using Corollaries~\ref{cor:5.2} and~\ref{cor:5.3}, the smallest representation of
$\Q[\ze_n]$ is of order $n$, whenever $a\ne 1$ and its order is
$\frac{n}{2}$, whenever $a=1$. The following theorem gives a $0,1$-symmetric
and circulant matrix
representation of order $n$ for  the largest subfield of $\Q[\zeta_n]$.

\begin{theorem}\label{thm:sym:represent}
Let $\de_n = \zeta_n+\zeta_n^{-1}$. \begin{enumerate} \item
\label{thm:sym:represent:1} Then $\mathbb{Q}[\de_n]$ has a symmetric $0,1$-circulant matrix representation of order $n$. \item
\label{thm:sym:represent:2} Let $\L$ be a subfield of $\mathbb{Q}[\de_n]$. Then there exists a symmetric circulant matrix that represents $\L$.
\end{enumerate}
\end{theorem}

\begin{proof}
Proof of part~\ref{thm:sym:represent:1}: Let $\K=\mathbb{Q}[\zeta_n]$.
Then $\Q[\de_n]$ is a subfield of $\K$ and $\zeta_n$ is a zero of the polynomial $x^2-\de_n x+1 \in \Q[\de_n][x]$. So, $[\K:\Q[\de_n]]=2$.
As $A = W_n + W_n^{-1}$, $A$ is a symmetric, $0,1$-circulant matrix. Also $\de_n$ is an eigenvalue of $W_n+W_n^{n-1}=W_n + W_n^{-1} = A$. Thus, by Theorem~\ref{thm:a10}, $A$ represents $\Q[\de_n]$. This completes the proof of Part~\ref{thm:sym:represent:1}.

Proof of part~\ref{thm:sym:represent:2}: Since $\L$ is a subfield of
$\Q[\de_n]$, there exists a polynomial $g(x) \in \Q[x]$ such that $\L =
\Q[g(\de_n)]$. So $\L=\Q[h(\ze_n)]$ where $h(x)=g(x+x^{n-1})$. Hence $\L$ can be
represented by $h(W_n)=g(W_n+W_n^{n-1})=g(A)$. Clearly, $h(W_n)$ is a symmetric, circulant matrix. Thus, the required result follows.
\end{proof}

We end this subsection, by a remark that gives an improvement on the order of the matrix $A$ of Theorem~\ref{thm:sym:represent}, whenever $n$ is an even integer. It is important to note that the representation given in the next remark, need not be a circulant representation. To do this, one uses a well known result that relates the eigenvalues of a cycle graph with the eigenvalues of a path graph.

\begin{rem}[Bapat, page~$27$, \cite{bapat}] \label{rem:even:cycle}
Let $n$ be an even positive integer and let $A$ denote the adjacency
matrix of the cycle $C_n$. Then the following results are well known:
\begin{enumerate}
\item $2$ and $-2$ are eigenvalues of  $A$.
\item Let $B_m$ be the adjacency matrix of the path $X_m$, on $m$ vertices.
 Then the set of eigenvalues of $B_{n/2 -
1}$ and the set of distinct eigenvalues of $A$, different from $2$ and $-2$,
are equal.
\end{enumerate}
Thus, the subfields of $\Q[\de_n]$ can also be represented by
 $g(B_{n/2-1})$, for some polynomial $g(x) \in \Q[x]$.
\end{rem}

\subsection{Representations of prime order}\label{sec:prime}

Let $p$ be a prime and let $\K$ be a subfield of $\Q[\ze_p]$. Then it is shown
in this subsection that there exists  a zero-one circulant matrix
$A$ of order $p$ such that the pair $(A,\J)$ represents $\K$.
To do this, we define Cayley graphs/digraphs.

\begin{definition}\label{def:cay}
Let $G$ be a group and let $S$ be a non-empty subset of $G$ that does not
contain the identity element of $G$. Then the {\em Cayley digraph/graph}\index{Cayley digraph} associated
with the pair $(G,S)$, denoted ${\mbox{Cay}}(G,S)$, has the set $G$ as its
vertex set and for any two vertices $x,y \in G, \; (x,y)$ is an edge if $x
y^{-1} \in S$.
\end{definition}

Observe that ${\mbox{Cay}}(G,S)$ is a graph if and only if $S$ is closed with
respect to inverse ($S = S^{-1} = \{ s^{-1} : s \in S \}$). Also, the graph is
$k$-regular if $S$ has $k$ elements. The set $S$ is called the connection set of
the graph and it can be easily verified that the graph ${\mbox{Cay}}(G,S)$ is
connected if and only if $G = \langle S \rangle$. We also recall that a  digraph
is called a {\em circulant digraph}\index{circulant digraph} if its adjacency matrix is a circulant
matrix. The next lemma, due to Biggs,  states that every circulant digraph can
be obtained as a Cayley digraph.

\begin{lemma}[Biggs~\cite{biggs}]\label{lem:cgroup}
Consider $\Z_n$ as a cyclic group of order $n$. Then every Cayley
digraph ${\mbox{Cay}}(\Z_n,S)$ is a circulant digraph. Conversely,
every circulant digraph on $n$ vertices is ${\mbox{Cay}}(\Z_n,S)$, for some non-empty subset $S$ of $\Z_n$.
\end{lemma}

We now  state a result due to Turner~\cite{turner} that relates the isomorphism of two circulant graphs of
prime order with their eigenvalues.

\begin{lemma}[Turner~\cite{turner}] \label{lem:turner}
Let $X_1$ and $X_2$ be  two circulant graphs of prime order. Then they are
isomorphic if and only if they have the same set of eigenvalues. Or
equivalently,  their connection sets are equivalent.
\end{lemma}

Before proving  a couple of results, we recall the following facts. These facts
are not stated in the present form but they can be obtained from
the results stated on  Pages~$554,577$ of Dummit $\&$ Foote~\cite{D:F}.

\begin{fact}[Dummit $\&$ Foote, Pages~$554, 577$~\cite{D:F}] \label{fact:unique}
Let $p$ be a prime. Then \begin{enumerate} \item \label{fact:unique:1} the
polynomial $f(x) = 1 + x + \cdots + x^{p-1}$ is irreducible over $\Q$.
\item \label{fact:unique:2} the Galois group of $\Q[\ze_p] $ over $\Q$ is
isomorphic to $\Z_p^*$, a cyclic group of order $p-1$. \item
\label{fact:unique:3}  for each divisor $d$ of  $p-1$, $\Z_p^*$ has a
unique subgroup of order $d$ and there exists a unique subfield of $\Q[\ze_p]$ whose degree of extension over $\Q$ is $d$.
\end{enumerate}
\end{fact}

\begin{lemma}\label{lem:a13}
Let  $p$ be a  prime number and let $k$  be any factor of $p-1$.
Then the edge set of  $K_p=(\Z_p,E)$, the complete graph on $p$ vertices, can be
partitioned into $k$  subsets $E_0,E_1, \ldots, E_{k-1}$ such
that the digraphs $X_i=(\Z_p,E_i)$, for $0 \le i \le k-1$, are $r$-regular circulant digraphs, where $r=\frac{p-1}{k}$. Moreover, the digraphs $X_i$ and $X_j$, for $0 \le i < j \le k-1$, are isomorphic.
\end{lemma}

\begin{proof}
 Let $\al$ be a generator of $\Z_p^*$. Then $H=\langle \alpha^{k}\rangle = \{1,
\al^k, \ldots, \al^{k(r-1)}\}$ is
a subgroup of $\Z_p^*$ having $r$ elements and let $H_j=\alpha^j H$, for  $j = 0,
1, \ldots, k-1$, be the cosets of $H$ in $\Z_p^*$ with $H_0 = H$.
It is important to note that $H_j$, as a subset of $\Z_p$, generates $\Z_p$, for
each $j = 0,1,\ldots, k-1$. Let us now define
a digraph $X_j$ by having $\Z_p$ as its vertex set  and  any two vertices $x, y \in \Z_p$, $(x,y)$ is an edge in $X_j$ if and only if $y-x \in H_j$. Then  it is easy to verify that $X_j$ is an $r$-regular Cayley digraph, ${\mbox{Cay}}(\Z_p,H_j)$. Also, observe that if we define   $A_j=\sum\limits_{h\in H_j}W_p^h$, for
$ 0 \le j \le k-1$, then $A_j$ is a $0,1$-circulant matrix and is the adjacency matrix of $X_j$.

Since the cosets $H_j$, for $0 \le j \le k-1$, are disjoint, one has
obtained $k$ disjoint digraphs that are $r$-regular and this completes the
proof of the first part.

We now need to show that the $k$ digraphs,
$X_j$, for $ 0 \le j \le k-1$, are mutually isomorphic. We
will do so by proving that the digraphs $X_0$ and $X_j$ are isomorphic,
for $1 \le j \le k-1$.

Let us define a map $\psi: \; V(X_0)\rightarrow V(X_j)$ by $\psi(s)=\al^{j}s$
for each $s \in V(X_0)$. Then it can be easily verified that $\psi$ is one-one
and onto. Thus, we just need to show that $\psi\bigl( (x,y)
\bigr) $ is an edge in $X_j$ if and only if  $(x,y) $ is an edge in $X_0$.
Or equivalently, we need to show that
$\psi(y) - \psi(x) \in H_j$ if and only if  $y-x \in H_0=H$. And this holds true
as $$y-x \in H \Leftrightarrow \al^j(y-x)\in H_j \Leftrightarrow (\al^j y-\al^j x)\in H_j \Leftrightarrow \psi(y)-\psi(x)\in H_j.$$ This completes the proof of the lemma.
\end{proof}

Before coming to the main result of this section, we have the following remark.
\begin{rem}\label{rem:hj} Let $p$ be a prime and let the cyclic group $\Z_p^* = \langle \al \rangle$ and its cosets $H_j = \al^j H$, for $0 \le j \le k-1$, be defined as in Lemma~\ref{lem:a13}. Then
\begin{enumerate}
\item \label{rem:hj:1}  using  Fact~\ref{fact:unique}.\ref{fact:unique:3},  the Cayley digraphs, $X_0, X_1, \ldots, X_{k-1}$,  constructed in the proof of Lemma~\ref{lem:a13} are unique.
\item \label{rem:hj:2} for a fixed $j$, $0 \le j \le k$, we observe the following.
\begin{enumerate}
\item \label{rem:hj:2:1} For each $h \in H, h H_j = H_j$. That is, for each $t, 0 \le t \le r-1$, $\al^{t k} H_j = H_j$. That is, $\al^{j} H = \al^{j + t k} H$, for all $t, 0 \le t \le r-1$.
\item \label{rem:hj:2:2} Let $\ze_p$ be a primitive $p$-th root of unity. Then, for each  $t, 0 \le t \le r-1$, $$\sum\limits_{h \in H} \bigl(\ze_p^{\al^j}\bigr)^h = \sum\limits_{s \in H_j} \ze_p^s = \sum\limits_{h \in H} \bigl(\ze_p^{\al^{j+ t k}}\bigr)^h.$$
\end{enumerate}
\end{enumerate}
\end{rem}
We are now ready to state and prove the main result of this section.

\begin{theorem} \label{thm:a14}
Let $p$ be a prime and let $\L$ be a subfield  of  $\Q[\zeta_p]$. Then
there exists a circulant digraph on $p$ vertices whose adjacency matrix
represents $\L$.
\end{theorem}

\begin{proof}
Let $ r =[\L: \mathbb{Q}]$. Then $r$ divides $p-1 = [\mathbb{Q}[\zeta_p]:
 \mathbb{Q}]$, as $\Q \subset \L \subset \mathbb{Q}[\zeta_p]$.  Let
$k=\frac{p-1}{r}$.   We now claim  the existence of a $0,1$-circulant matrix
$A$ of order $p$ whose minimal polynomial has an irreducible factor of degree
$r$.

As $k$ divides $p-1$,  Lemma~\ref{lem:a13}, gives us a collection, say $X_0, X_1,
\ldots, X_{k-1}$,  of $r$-regular circulant digraphs on $p$ vertices that are
mutually isomorphic. Let $A$ be the adjacency matrix of $X_0$. Then, using the
definition of $X_0$, its adjacency matrix  $A=\sum\limits_{h\in H} W_p^h$. Thus, $A$ is
a circulant matrix and hence diagonalizable. Thus, we just need to find the
eigenvalues of $A$ to get the minimal polynomial of $A$.  By  definition,  the eigenvalues of $A$ are $\lambda_i=
\sum\limits_{h\in H} (\ze_p^i)^{h}$, for $0 \le i \le p-1$. Observe
that $|\la_i| \le r$ and $\la_i = r$ if and only if $i = 0$. Fix an $i \in \{1, 2, \ldots, p-1\}$. Then, $i \in H_j$, for some coset $H_j$, $0 \le j \le k-1$, of $\Z_p^*$. Therefore, using Remark~\ref{rem:hj}.\ref{rem:hj:2}, we see that
$\la_{\al^j} = \la_{\al^{j+k}} = \cdots = \la_{\al^{j+ (r-1) k}}$, for each $j \in \{0, 1, 2, \ldots, k-1\}$. That is, for each fixed $j \in \{0, 1, \ldots, k-1\}$ and  $s, t \in H_j$, $\la_s = \la_t$.
Thus,  $A$ has exactly $k$ distinct eigenvalues other than  the eigenvalue $r$.
Also, note that $A$ is  a circulant matrix of order $p$, a prime.
Therefore, by Theorem~\ref{thm:a10}.\ref{thm:a10:2}, the minimal
polynomial of $A$ factors into two distinct irreducible factors. One of
the factor is  $x-r$, corresponding to the simple eigenvalue $r$ of $A$ and the
other must contain all the distinct eigenvalues of $A$, different from $r$.
Hence, the minimal polynomial of $A$ equals $(x-r) \prod\limits_{i=1}^{k}
(x - \la_i) = (x-k) q(x) \in \Q[x].$

As $\deg(q(x)) = k$, the $0,1$-circulant matrix $A$ represents a subfield, say
$\K$, of $\Q[\ze_p]$ such that $[\K: \Q[\ze_p] ] = k.$ Thus, the proof of the
claim is complete.

Now, using Fact~\ref{fact:unique}.\ref{fact:unique:3}, the subfield $\K$ is
indeed the subfield $\L$.
\end{proof}

We have seen that if $p$ is a prime then $\Q[\ze_p]$ has a unique subfield
for each divisor $d$ of $p-1$. But all the real subfields of $\Q[\ze_p]$
are also subfields of $\mathbb{Q}(\zeta_p+\zeta_p^{-1})$. This observation
leads to the last result of this section.

\begin{cor}
Let $p$ be a prime number. Then every real subfield of
$\mathbb{Q}[\zeta_p]$ has a symmetric $0,1$-circulant matrix
representation of order $p$.
\end{cor}

Let $p$ be a prime and consider the digraph $X_0$ in the proof of
Theorem~\ref{thm:a14}. Since $p$ is a prime, it can be easily verified that
$X_0$ is a strongly connected regular digraph.   Hence, using
Lemma~\ref{lem:hoff}, one immediately obtains  the following result and hence
the proof is omitted.

\begin{cor} \label{a15}
Let  $\L$ be  a subfield  of $\mathbb{Q}[\zeta_p]$. Then there exists a
$0,1$-circulant matrix $A$ of order $p$ such that $(A,\J)$ represents
$\L$.
\end{cor}

\section{Smallest $0,1$- Companion Matrices Whose Minimal Polynomial is Divisible
by $\Phi_n(x)$}
\label{sec7}

In this section, for a fixed positive integer $n$, our objective is
to find a $0,1$-companion matrix of the smallest order that
represents $\Q[\ze_n]$. Let $\alpha$ denote the generic element such
that $\Q[\alpha]= \Q[\zeta_n]$. Using Corollary~\ref{c1}, this is
equivalent to finding the smallest $0,1$-companion matrix whose
minimal polynomial is divisible by the minimal polynomial of
$\alpha$. As there are infinitely many choices for  $\alpha$, we
restrict ourselves to $\alpha=\ze_n$. Hence,  we search for a
polynomial $f(x) \in \Z[x]$ of least degree such that $\Phi_n(x)$
divides $f(x)$ and $\Cl(f)$, the companion matrix of $f(x)$, is a
matrix with entries $0$ and $1$.  A similar study was made by  Filaseta $\&$ Schinzel~\cite{M:A} and Steinberger~\cite{J:S}, where they looked at polynomials with integer coefficients that are divisible by $\Phi_n(x)$.

Let $f(x) = x^n- a_{n-1}x^{n-1}-\dots -a_1x-a_0 \in \Z[x]$. Then $\Cl(f)$ is a
$0,1$-matrix if and only if for each $i, 0 \le i \le n-1, \; a_i\in \{0,1\}$.
Since $gcd(x^k,\Phi_n(x)) = 1$ for all $k\geq 1$, without loss of generality, we
can assume that  $a_0=1$. By definition, $\Phi_n(x)$ divides $g(x) = x^n-1$.
Hence $\Cl(g)$ is a $0,1$-matrix of order $n$ that represents $\Q[\zeta_n]$.
In order to determine whether there exists a matrix of smaller order, we define
a set $\A_n$ as \begin{eqnarray}
\A_n = \{ f(x) \in \Z[x] & : & \Phi_n(x) \mid f(x), f(x) \equiv x^m - a_{m-1}
x^{m-1} - \cdots a_1 x-1 \nonumber \\ && \hspace{.5in}  m < n, a_i \in \{0, 1\}
{\mbox{ for }} 1 \le i < m \} \label{defn:an}
\end{eqnarray}
and try to find the polynomial of least degree in $\A_n$.

Let $f(x) \in \A_n$. Then $f(x)$ has at least three terms as $m < n$. Hence,
$f(1)\ne 0$. Now, let $p$ be a prime. Then $\deg(f(x)) > \deg(\Phi_p(x)) =p-1$.
Thus, $\A_p $ is an empty set.  This is stated as our next result.


\begin{lemma}\label{lem:prime}
Let $p$ be a prime number. Then $\A_p$ is an empty set.
\end{lemma}

\begin{rem}\label{rem:Ap}
Let $p$ be a prime. Then, starting with the field $\Q[\zeta_p]$, a $0,1$-matrix representing $\Q[\zeta_p]$ of least order is $W_p$, the companion matrix of $x^p - 1$. But, it can be easily seen that if there exists a $0,1$-matrix $A \in \M_{\ell}(\C)$ representing $\Q[\al] \cong \Q[\zeta_p]$ then $\ell \ge p-1$. Thus, it may be possible to get a $0,1$-matrix $A \in \M_{p-1}(\C)$ such that $A$ represents  $\Q[\al] \cong \Q[\zeta_p]$. 
\end{rem}

We now state a well known result about cyclotomic polynomials which enables us  to consider only  square-free positive integers $n$, where recall that  a positive integer $n$  is said to be square free if the decomposition of $n$ into primes
does not have any repeated factors.

\begin{lemma} [Prasolov, Page:93~\cite{V:P}] \label{lem:cyclo}
Let p be a prime number and let $n$ be a positive integer. Then
$$\Phi_{pn}(x)=
\begin{cases} \Phi_n(x^p), & \mbox{ if } p\mid n,\\
\frac{\Phi_n(x^p)}{\Phi_n(x)}, & \mbox{ if } p\nmid n.
\end{cases}$$
In particular, if   $n=p_1^{a_1} \cdots p_k^{a_k}$ is a prime
factorization of $n$ into distinct primes $p_1, p_2, \ldots, p_k$ and if
$n_0 = p_1 p_2 \cdots p_k$ then $\Phi_n(x)=\Phi_{n_0}(x^{n/n_0})$.
\end{lemma}

Steinberger~\cite{J:S} pointed out that
the problem of finding polynomials divisible by $\Phi_n(x)$ is
equivalent to finding polynomials divisible by $\Phi_{n_0}(x),$
where $n_0$ is the maximum square-free factor of $n$. Following is
a similar assertion in the current context. We give the proof for the sake of
completeness.

\begin{lemma}\label{lem:squarefree}
Let $n = p_1^{a_1}p_2^{a_2}\cdots p_k^{a_k}$ be a factorization of $n$
into distinct primes $p_1, p_2, \ldots, p_k$ and let $n_0=p_1p_2 \cdots p_k$.
Then
$$\min\{\deg(f(x)): f(x) \in \A_n\}=\frac{n}{n_0}\min\{\deg(f(x)): f(x)
\in \A_{n_0}\}.$$
\end{lemma}

\begin{proof}
Let $f(x)\in \A_{n_0}$. Then by  Lemma~\ref{lem:cyclo},  $f(x^{n/n_0})\in
\A_n$.

Conversely, suppose $f(x)\in \A_n$. Then $\Phi_n(x)$
divides $f(x)$ and therefore using Gauss lemma on polynomials [see~Dummit $\&$ Foote, Page~$304$ \cite{D:F}] and  Lemma~\ref{lem:cyclo}, $$f(x) = \Phi_n(x)
g(x) = \Phi_{n_0}(x^{n/n_0}) g(x) {\mbox{ for some }} g(x) \in \Z[x].$$ We now
group the terms of
$g(x)$ such that $g(x)=\sum\limits_{i=0}^{\frac{n}{n_0}-1}g_i(x^{n/n_0})
x^i$, where $g_i(x^{n/n_0})$ is a polynomial in $x^{n/n_0}$ (collect the
terms containing the exponents that are equivalent to $i \pmod {n/n_0}$).
Therefore,
$$f(x) = \sum_{i=0}^{\frac{n}{n_0}-1} \Phi_{n_0}(x^{n/n_0}) g_i(x^{n/n_0})
x^i = \sum_{i=0}^{\frac{n}{n_0}-1} f_i(x) x^i \; {\mbox{(say)}}.$$ That
is, the polynomials $f_i(x)$, for $0 \le i \le n/n_0 -1$, are divisible by
$\Phi_n(x) = \Phi_{n_0}(x^{n/n_0})$. Let the polynomial $f_j(x) x^j$ contain
$x^m$, the leading term of $f(x)$. Then $f(x) \in \A_n$ implies that
$$\Phi_{n_0}(x^{n/n_0}) g_j(x^{n/n_0})
x^j = x^m-x^{r_\ell}-x^{r_{\ell-1}}-\cdots-x^{r_1}, \; {\mbox{ with }} j
\le  r_1 < r_2 < \cdots < r_\ell.$$
 As $\gcd(\Phi_n(x),x^{r_1})=1$, the
polynomial $h(x)=x^{m-r_1}-x^{r_\ell-r_1}-\cdots-x^{r_2-r_1}-1$ is expressible
as a polynomial in $x^{n/n_0}$ and is divisible by $\Phi_{n_0}(x^{n/n_0})$.
Hence, one obtains a polynomial
$h_1(y)=y^{m'}-y^{r'_\ell}-\cdots-y^{r'_2}-1\in \A_{n_0}$ such that
$h(x)=h_1(x^{n/n_0})$ and $\frac{n}{n_0}\deg(h_1)\leq \deg(f)=m$. Thus, the
desired result follows.
\end{proof}

\begin{rem}\label{rem:squarefree}
Lemma~\ref{lem:squarefree} implies that in order to determine
$\min\{\deg(f(x)): f(x) \in \A_n\}$,  it is sufficient to solve the same problem
in $\A_{n_0}$, where $n_0$ equals the product of  all the prime factors
of $n$, a square-free positive integer. Henceforth, $n$  will be a square-free
positive integer.
\end{rem}

Lemma~\ref{lem:prime} together with Lemma~\ref{lem:squarefree} leads to
our next result.

\begin{lemma}\label{lem:pk}
Let $p$ be a prime and let $n=p^k$ for some $k\in\Z^+$. Then $\A_n$ is an empty set.
\end{lemma}

Thus, we will be interested only in those positive integers $n$ that has at
least two prime factors. In this case, it will be shown (see
Corollary~\ref{cor:pq} on Page~\pageref{cor:pq}) that the set $\A_n$ in
non-empty. To start with, note that  \begin{equation}
\varphi(n)\leq \min\{\deg(f(x)) : \; f(x) \in \A_n \} < n. \label{eqn:bound:an}
\end{equation}

Using a small observation, we improve the lower-bound in
Equation~(\ref{eqn:bound:an}) as follows.

\begin{lemma}\label{lem:lower}
Let $n$ be a positive integer. Then
\begin{equation}
\label{eqn:bound}\max \{\varphi(n), \lceil\frac{n}{2} \rceil \} <
\min\{\deg(f(x)) : \; f(x) \in
\A_n \} < n.\end{equation}
\end{lemma}

\begin{proof}
The lemma is immediate from Equation~(\ref{eqn:bound:an}) if we can show that
$\deg(f(x)) >\lceil\frac{n}{2} \rceil$. Suppose
$f(x)=x^m-x^{k_{\ell}}-x^{k_{\ell-1}}-\dots-x^{k_1}-1 \in \A_n$ with $0<
k_1<k_2< \cdots < k_\ell<m$. As $ \Phi_n(x)$ divides $f(x)$,  $f(\ze_n) = 0$.
Now, let if possible, $m\leq\frac{n}{2}$. Then $n - 2m \ge 0$ and using the fact
that $\ze_n = \ze_{2n}^2 = \cos(2 \pi/n) + i \sin(2 \pi/n)$ and $\ze_{2n}^n =
-1$, we get
\begin{eqnarray*}
0 = f(\ze_n) &=&
-\ze_n^m+\zeta_n^{k_\ell}+\zeta_n^{k_{\ell-1}}+\cdots+\zeta_n^{k_1}+1 =
-\ze_{2n}^{2m}+\zeta_{2n}^{2 k_\ell} + \zeta_{2n}^{2 k_{\ell-1}} + \cdots +
\zeta_{2n}^{2 k_1}+1 \\
&=& 1+\ze_{2n}^{n-2m+2k_{\ell}} +  \ze_{2n}^{n-2m+2k_{\ell -1}} + \cdots +
\ze_{2n}^{n-2m+2k_1}+ \ze_{2n}^{n-2m} \\
&=& 1 + \sum_{j=0}^\ell \left( \cos\left(\frac{(n-2m + 2 k_j) \pi}{n}\right) + i
 \sin\left(\frac{(n-2m + 2 k_j) \pi}{n}\right)\right),
\end{eqnarray*}
 where $k_0=0$. Now using the choice of $k_j$'s, one gets $n - 2m + 2 k_j < n$
for each $j=0, 1,2, \ldots, \ell$. Hence, $\sum\limits_{j=0}^\ell
\sin(\frac{(n-2m + 2 k_j) \pi}{n})$ cannot be zero. Thus, we have arrived at a
contradiction and therefore the required result follows.
\end{proof}

This section is arranged as follows: the first subsection is devoted to
characterizing $\A_n$ in terms of certain subsets of roots of unity.  In
Subsection~\ref{sec2:2:1}, the exact value of $\min\{\deg(f(x)): f(x) \in
\A_n\}$ is obtained whenever $n$ has exactly two prime factors. The last
subsection, namely Subsection~\ref{sec2:2:2}, gives a bound on
$\min\{\deg(f(x)): f(x) \in \A_n\}$ whenever $n$ has $3$ or more prime factors.

\subsection{Characterization of $\A_n$ by One-Sums}
 \label{sec:An-Bn}

For a fixed positive integer $n$,  let $U_n = \{k : 1 \le k \le n, \gcd(k,n) =
1\}$ and  $R_n = \{\ze_n^k : \; 0 \le k \le n-1
\}$. Then $| U_n| = \varphi(n)$, $R_n$ contains all the $n$-th roots of unity
and $\{\ze_n^k : k \in U_n \}$ contains all the primitive $n$-th roots of unity.
 The following result can be found in Apostol~\cite{T:A}.

\begin{lemma}[Apostol~\cite{T:A}]\label{lem:mu}
Let $n$ be a positive integer. Then $\sum\limits_{k \in U_n}  \ze_n^k =
\mu(n)$, where $$ \mu(n)=
\begin{cases} 0, & {\mbox{ if }} \; n \; {\mbox{ is not square free,}} \\
 1, &  {\mbox{ if }} \; n \; {\mbox{ has even number of prime factors,}}\\
-1, & {\mbox{ if }} \; n \; {\mbox{ has odd number of prime factors.}}
\end{cases}$$
\end{lemma}

Fix a positive integer $n$ and let $T \subset R_n$. Let $\sigma(T) =
\sum\limits_{\al \in T} \al$, denote the sum of all the elements of $T$. In
particular, recall that $\sigma (R_n)
= 0$. We now define a subset $\Bl_n $ of $R_n$ by
 \begin{equation}\label{defn:bn}
 \Bl_n = \{ T \subset R_n\setminus \{1\} : \; \sigma(T) = 1 \}.
\end{equation}
Then the next result gives a bijection
between the sets $\A_n$ and $\Bl_n$. This correspondence is useful in
constructing  members of $\A_n$.

\begin{theorem}\label{thm:An-Bn}
Let $\A_n$ and $\Bl_n$ be defined as above. Then there exists a bijection
 between  $\A_n$ and
$\Bl_n$ such that $x^m-1-x^{k_1}-x^{k_2}-\cdots -x^{k_l}\in \A_n$ corresponds to
$\{\zeta^{n-m},\zeta^{n-m+k_1},\cdots ,\zeta^{n-m+k_l}\}\in \Bl_n$.
\end{theorem}

 \begin{proof}
Let $f(x)=x^m-1-x^{k_1}-x^{k_2}-\dots-x^{k_\ell} \in \A_n$ with $1\le
k_1<k_2< \cdots < k_\ell<m< n$. As $f(x) \in \A_n,  f(\ze_n) = 0$ and hence
$\zeta_n^m=1+\zeta_n^{k_1}+\zeta_n^{k_2}+\cdots+\zeta_n^{k_\ell}.$
Or equivalently, $T = \{\zeta_n^{n-m},\zeta_n^{n - m +
k_1},\ldots,\zeta_n^{n-m+k_\ell}\}\in \Bl_n \; {\mbox{ as }} \; \sigma(T) = 1.$

Conversely, let $T = \{\ze_n^{k_0}, \ze_n^{k_0+k_1},
\ze_n^{k_0+k_2}, \ldots, \ze_n^{k_0+k_\ell} \} \in \Bl_n$,
where  $1\le k_0< k_0+k_1<k_0+k_2< \cdots < k_0+k_\ell < n$. Then
$\sum\limits_{i=1}^{\ell} \ze_n^{k_0+k_i} + \ze_n^{k_0} = 1$, or equivalently,
 $\ze_n^{n-k_0} = 1 + \ze_n^{k_1} + \cdots + \ze_n^{k_\ell}=0.$ Thus,  $f(x) =
x^{n-k_0} - x^{k_\ell}- x^{k_{\ell-1}}
\cdots - x^{k_1} - 1\in \A_n$ and  the required result
follows.
 \end{proof}

Theorem~\ref{thm:An-Bn} leads to the following important remark.

\begin{rem}\label{rem:min:degree}
Fix a positive integer $n$ and let $f(x)$ be a polynomial of least degree
in $\A_n$. Then $\deg(f(x)) = n- k_0$, where $k_0$ is obtained as
follows: ``for each element $T$ of $\Bl_n$, let $k_T$ be the least
positive integer such that $\ze_n^{k_T} \in T.$ Then $k_0 = \max\{k_T : T
\in \Bl_n\}$".
\end{rem}

We now observe the following. Let $n$ be a positive integer
and let $d$ be the product
of an even number of distinct prime divisors of $n$. Also, let us write $
\ze_n^{n/d} = \ze_d.$ Then using Lemma~\ref{lem:mu},  $\{ \ze_d^k : \; k \in
U_d\} \in \Bl_n$.
Observe that $U_d=\{1,1+k_1,1+k_2,\dots,1+k_\ell = d-1\}$ for some $k_i$'s
satisfying $1 \le k_1 < k_2 < \cdots < k_\ell = d-2$.
Therefore, $\ze_d^{d-1}= \ze_d^{-1} = 1 + \ze_d^{k_1} + \cdots +
\ze_d^{k_{l-1}}+
\ze_d^{d-2}$ and hence
$$f(x) = x^{\frac{n}{d}(d-1)}
-x^{\frac{n}{d}(d-2)}-x^{\frac{n}{d}(k_{\ell-1})}-\dots -x^{\frac{n}{d}(k_1)}
-1 \in \A_n$$ is the corresponding polynomial.
Note that $\deg(f(x)) = n - \frac{n}{d}.$ This observation leads to the
first part of the following result. The second part directly follows from
the first part and hence the proof is omitted.

\begin{cor}\label{cor:pq}
Let $n = p_1^{a_1}p_2^{a_2}\cdots p_k^{a_k}$ be a factorization of $n$
into distinct primes and let $d$ be the product of an even number of
distinct prime divisors of $n$. If $U_d=\{1,1+k_1,1+k_2,\ldots,
1+k_\ell\}$ with $1 \le k_1 < k_2 < \cdots < k_\ell = d-2$, then
\begin{enumerate}
\item \label{cor:pq:1}$\Phi_n(x)$ divides the polynomial
$f(x) = x^{\frac{n}{d}(d-1)} -x^{\frac{n}{d}(d-2)}-x^{\frac{n}{d}(k_{\ell-1})}
-\dots -x^{\frac{n}{d}(k_1)}-1.$
\item \label{cor:pq:2} $\min\{ \deg(f(x)) : \; f(x) \in \A_n
\} \le n- \dfrac{n}{p_1 p_2}$, where $p_1$ and $p_2$ are the two smallest
prime divisors of $n$.
\end{enumerate}
\end{cor}

\subsection{Integers with Two Prime Factors} \label{sec2:2:1}

Let $n = p_1^{a_1} p_2^{a_2}$ be the factorization of $n$ as product of two
distinct primes $p_1$ and $p_2$. Then it is shown that the upper bound obtained
in Corollary~\ref{cor:pq} is indeed attained. That is,
$\min\{ \deg(f(x)) : f(x) \in \A_{n} \} = \frac{n}{p_1 p_2}\left\{p_1
p_2-1\right\}$.

Before proceeding further,
recall that for any positive integer $n$ and non-negative integer $m$, the
{\em Ramanujan's sum}\index{Ramanujan's sum}  is defined as $c_n(m) =
\sum\limits_{k \in U_n} (\ze_n^k)^m.$ The next lemma is  a well known result
related with the Ramanujan's sum (for results related with Ramanujan's sum and coefficients of cyclotomic polynomials, see  Moree $\&$ Hommerson~\cite{mor}).

\begin{lemma}[Moree $\&$ Hommerson~\cite{mor}]\label{lem:ramanujan}
Fix  positive integers $m$ and $n$.  Then, for each divisor   $d$ of $n$,
$c_n(d)=\mu(\frac{n}{d})\frac{\varphi(n)}{\varphi(\frac{n}{d})}$.
Furthermore, $c_n(m)=c_n(d)$ whenever $\gcd(m,n)=d$.
\end{lemma}

Let $m < n$ be a positive integer. Then Ramanujan's sum is used to assign a number to  a polynomial
$g(x)=\sum\limits_{i=0}^m a_i x^i \in \Q[x]$ via the sum $\sum_{k\in U_n} g(\ze_n^k)$, denoted $S_g$. Then
\begin{equation}\label{eqn:sg}
S_g=\sum_{i=0}^m a_i c_n(i)=a_0\varphi(n)+
\sum\limits_{d \mid n} \left( \sum\limits_{i\in U_d}a_{ni/d}
\right)\mu(d)\frac{\varphi(n)}{\varphi(d)}. \end{equation}
Since, $\Phi_n(x)$ divides $f(x)$, $f(\ze_n^k) = 0$
for each $k \in U_n$. Thus, the next result is immediate  and hence the proof is
omitted.

\begin{lemma}\label{lem:Sf}
Let  $n$ be a positive integer. Then for each $f(x) \in \A_n$, $S_f = 0$.
\end{lemma}

Therefore, for any  $f(x) \in \Q[x]$, $S_f = 0$ gives a necessary
condition for $\Phi_n(x)$ to divide $f(x)$.
The next result is the main result of this subsection and it is shown that
$\min\{ \deg(f(x)) : f(x) \in \A_{p_1 p_2} \} =
p_1 p_2-1$, whenever $p_1$ and $p_2$ are distinct primes. This result together
with Lemma \ref{lem:squarefree}  implies that if $p_1$ and $p_2$ are distinct
primes and $n=p_1^{a_1} p_2^{a_2}$, for some positive integers $a_1$ and $a_2$,
then $\min\{ \deg(f(x)) : f(x) \in \A_{n} \} = \frac{n}{p_1 p_2}
\left\{p_1 p_2-1\right\}.$

\begin{theorem} \label{thm:p1p2}
Let $p_1$ and $p_2$ be two distinct primes. Then
$$\min\{ \deg(f(x)) : f(x) \in \A_{p_1 p_2} \} = p_1 p_2-1.$$
\end{theorem}

\begin{proof}
Let $n = p_1 p_2$. Then using a contrapositive argument, we will first show that
$\min\{ \deg(f(x)) : f(x) \in \A_{n} \} \ge
n-1$.  Let $f(x) \in \A_{n}$ be the
polynomial of least degree with $\deg(f(x)) < n-1$. Then
Theorem~\ref{thm:An-Bn} gives the existence of a subset $T = \{ \ze_n^{k_1},
\ze_n^{k_2}, \ldots, \ze_n^{k_\ell}\}$ of $\Bl_n$ with $2 \le k_1 < k_2 < \cdots
< k_\ell$ that corresponds to $f(x)$. Define $g(x)
= \sum\limits_{i=1}^\ell x^{k_i} - 1. $ Then $g(x) \in \Z[x]$ and $g(\ze_n) =
0$. Thus, for all $k \in U_n, g(\ze_n^k)=0$ and $S_g = 0$.

Now, for each divisor $d$ of $n$, define $N_d = \{i \frac{n}{d}  : \; i\in U_d\} \cap \{k_1, k_2,
\ldots, k_\ell\}$. Then, using Equation~(\ref{eqn:sg}), one has $0 = S_g =
\sum\limits_{d \mid n} |N_d| \; \mu(d) \;\frac{\varphi(n)}{\varphi(d)} -
\varphi(n).$ Or equivalently, $\varphi(n) = \sum\limits_{d \mid n} |N_d| \;
\mu(d) \;\frac{\varphi(n)}{\varphi(d)}.$ Therefore, using
$\mu(p_i) = -1$ for $i=1,2$ and $\mu(p_1 p_2) = 1$, one gets
$$\frac{|N_n|}{\varphi(n)}=1+ \frac{|N_{p_1}|}{\varphi(p_1)}+
\frac{|N_{p_2}|}{\varphi(p_2)}$$
as $|N_1|=0$.
But observe that $|N_n| < \varphi(n)$ as $k_1 \ge 2$. That is, the left hand
side of the above identity is less than $1$ which contradicts the expression
that appears on the right hand side. Thus, our assumption is not valid and
hence $\min\{ \deg(f(x)) :
f(x) \in \A_{n} \} \ge n-1$.

As $n = p_1 p_2$, Corollary~\ref{cor:pq}.\ref{cor:pq:2} implies that
$\A_n$ is non-empty and hence
$\min\{ \deg(f(x)) : f(x) \in \A_{n} \} \le n-1$. Thus, the
required result follows.
 \end{proof}

\subsection{Even Integers with $3$ or more Prime Factors}
\label{sec2:2:2}

In this subsection, we improve the bound
given in  Corollary~\ref{cor:pq} for all even positive integers that have more
than $2$ prime factors.

\begin{theorem}\label{thm:main}
Let $ p_1<p_2<\cdots <p_k$ be odd primes and let $n= 2 p_1 p_2\cdots p_k$. Then
$\min\{ \deg(f(x)) : f(x) \in \A_{n} \}\leq n-v$ where $$v=
\begin{cases}
\frac{n}{2} \cdot \frac{p_1+p_2}{p_1p_2}, & {\mbox{ if }}\; 2p_1>p_2,\\
\frac{3n}{2p_2},  & {\mbox{ if }}\; 2p_1<p_2<3p_1,\\
\frac{n}{2p_1}, & {\mbox{ if }}\; 3p_1<p_2.
\end{cases}$$
\end{theorem}

\begin{proof}
Let $f_0(x)$ be the polynomial of least degree in $\A_n$.  We will find numbers
$v_1$ and $v_2$, as lower bounds for $n-\deg(f_0(x))$ and take
$v=max\{v_1,v_2\}$. The value of $v_1=\frac{n}{2 p_1}$ is a direct application
of
Corollary~\ref{cor:pq}  as $2$ and $ p_1$ are the smallest two prime
divisors of $n$. Now, let us compute $v_2$.

To get the value of $v_2$, consider  $T = \{ \ze_n^{n r/{2 p_1}} : \; r \in
U_{2 p_1} \} \cup \{ \ze_n^{n \ell/{p_2}} : \; \ell \in U_{p_2} \}.$ Then using
Lemma~\ref{lem:mu},
$ \sum\limits_{z \in T} z = \sum\limits_{r \in U_{2 p_1}}  \ze_n^{n r/{2
p_1}} + \sum\limits_{\ell \in U_{p_2}} \ze_n^{n \ell/{p_2}} = 1 + (-1) =
0.$ Multiplying both sides by $\ze_n^{n/{2 p_2}}$ and observing that
$\left(\ze_n^{n/{2 p_2}}\right)^{p_2} = -1 $ (as $p_2$ is an odd prime), one gets
\begin{eqnarray}
0 &=& \sum_{r \in U_{2 p_1}}  \ze_n^{n r/(2 p_1) + n/(2 p_2)} + \sum_{\ell
\in U_{p_2}} \ze_n^{n \ell/{p_2} + n/(2 p_2)} \nonumber \\
&=&  \sum_{r \in U_{2 p_1}} \ze_n^{n r/(2 p_1) + n/(2 p_2)} + \sum_{\ell
\in U_{p_2},  2 \ell < p_2 -1 } \ze_n^{n \ell/{p_2} + n/(2 p_2)}
+\ze_n^{\frac{n}{2}}
\nonumber\\
&& \hspace{.5in} + \sum_{\ell \in U_{p_2},  2 \ell > p_2 - 1 } \ze_n^{n
\ell/{p_2} + n/(2
p_2)}  \nonumber \\
&=&  \sum_{r \in U_{2 p_1}} \ze_n^{n r/(2 p_1) + n/(2 p_2)} + \sum_{\ell
\in U_{p_2},  2 \ell < p_2 -1 } \ze_n^{n \ell/{p_2} + n/(2 p_2)} -1\nonumber \\
&& \hspace{.5in} + \sum_{\ell \in U_{p_2},  2 \ell > p_2 - 1 } \ze_n^{n
\ell/{p_2} + n/(2
p_2)}. \label{eqn:v2}
\end{eqnarray}
Thus, Equation~(\ref{eqn:v2}) implies
$$T' = \{ \ze_n^{n r/(2 p_1) + n/(2 p_2)} : r \in U_{2 p_1}\} \cup \{ \ze_n^{n
\ell/{p_2} + n/(2
p_2)} : \ell \in U_{p_2}\setminus \{(p_2-1)/2\} \}  \; \in \Bl_n.$$ That is,
$v_2 =\min\{ r : \ze_n^r \in T' \} =
\begin{cases}
\frac{n}{2} \left( \frac{p_1+p_2}{p_1 p_2} \right), & {\mbox{ if }}\; 2p_1>p_2,
\\
\frac{3n}{2p_2}, & {\mbox{ if }}\; 2p_1<p_2.
\end{cases}$

Hence, using Remark~\ref{rem:min:degree}  the required result follows.

\end{proof}

\subsection{When $n$ is Even and $\Phi_n(x)$ is Flat}
In this subsection, the upper bound for $\min\{ \deg(f(x)) : f(x) \in \A_{n} \}$ is improved further  whenever $n$ is even  and the cyclotomic polynomial $\Phi_n(x)$ is {\em flat}\index{flat polynomial}. To do so, recall that the {\em height}\index{height of a polynomial} of a polynomial in $\Z[x]$ is the largest absolute value of its coefficients and a polynomial is said to be {\em flat} if its height is $1$.
 Let $A(n)$ be the height of $\Phi_n(x)$.
It is known that for all $n < 105$,
$\Phi_n(x)$ is flat and height of $\Phi_{105}(x)$ is $2$. In
fact, the height of $\Phi_n(x)$ is unbounded [see Emma Lehmer~\cite{lehmer}].

Let $k$ be the number of distinct odd
prime factors of $n$. For square-free $n$, this number $k$ is
called the  {\em order}\index{order of a cyclotomic polynomial} of the cyclotomic polynomial
$\Phi_n(x)$.  It is known that all cyclotomic polynomials of order
$1$ and order $2$ are flat.
Gennady Bachman~\cite{gb} gave the first infinite family of flat cyclotomic
polynomials of order three
 and this family was expanded by Kaplan~\cite{kap1}. In \cite{kap2}, Kaplan gave
some flat
polynomials of order four. It is unknown whether there
are any flat cyclotomic polynomials of order greater than four.

Fix a positive integer $k$ and let  $n= 2 p_1 p_2 \cdots p_k$, for distinct odd primes $p_1<p_2<\cdots < p_k$.
Let  $\Phi_n(x)=\prod\limits_{i=1}^{\varphi(n)}(x-x_i) =\sum\limits_{t=0}^{\varphi(n)}(-1)^t e_t x^{\varphi(n)-t}$, where  $x_1, x_2,\ldots, x_{\varphi(n)}$ are distinct roots of $\Phi_n(x)$ and
$e_t=\sum\limits_{1\leq i_1<i_2<\dots<i_t\leq \varphi(n)}\prod\limits_{j=1}^tx_{i_j}$. Then it is known that $e_t=e_{\varphi(n)-t}$, for $0\leq t\leq \varphi(n)$ and  $e_0=1$ [see Thangadurai~\cite{tan}]. Further,  by  Newton-Girard formulas
\begin{equation}\label{eqn:sum}
me_m=e_{m-1}c_n(1)-e_{m-2}c_n(2)+\dots + (-1)^me_1c_n(m-1)+(-1)^{m-1}c_n(m).
\end{equation}
where $c_n(m)$ is the Ramanujan's sum defined in Page~\pageref{lem:ramanujan}. In particular, using Lemma~\ref{lem:mu} $e_1=c_n(1)= \mu(n) = -1$.

Now let $k$ be an even integer. That is, $n$ is product of odd number of distinct primes.   Then for any positive integer $m<p_1$,
\begin{equation*}
c_n(m)=
\begin{cases}
- 1, &
{\mbox{ if }} \; m \; {\mbox{ is odd,}} \\
1, &
{\mbox{ if }} \; m \; {\mbox{ is even. }}
\end{cases}
\end{equation*}
Now, using Equation~(\ref{eqn:sum}) recursively, it is easy to show that $e_2=\cdots=e_{p_1-1}=0$ and $e_{p_1}=1$.

With these observations, we have
\begin{equation}\label{eq:Phi}
\Phi_n(x)= \begin{cases} x^{\varphi(n)}-x^{\varphi(n)-1}\pm \dots -x+1, &
{\mbox{ if }} \; k \; {\mbox{ is odd,}} \\
x^{\varphi(n)}+x^{\varphi(n)-1}-x^{\varphi(n)-p_1}\pm \dots -x^{p_1}+x+1, &
{\mbox{ if }} \; k \; {\mbox{ is even.}}
\end{cases}
\end{equation}

From now on, we consider only flat cyclotomic polynomials. Then $\Phi_n(x)=f_1(x)-f_2(x)$, for some $0,1$-polynomials $f_1(x)$ and $f_2(x)$. Observe that  the representation of $\Phi_n(x)$ as difference of two $0,1$-polynomials is unique. Also,
$ \Phi_n(\ze_n)=0$ implies that $  f_1(\ze_n)-f_2(\ze_n)=0$ and hence $ f_1(\ze_n)+\ze_n^{n/2}\cdot f_2(\ze_n)=0.$  That is, $ \Phi_n(x)$ divides  $f_1(x)+x^{n/2}f_2(x)$.

Let $\Phi^T_n(x)= f_1(x)+x^{n/2}f_2(x)$. Then $\Phi^T_n(x)$ is a
$0,1$-polynomial and $\Phi_n^T(\ze_n) = 0$.
And from Equation~(\ref{eq:Phi}), we have
\begin{equation}
 \deg(\Phi^T_n(x))=\begin{cases} \phi(n)-1+\frac{n}{2},\;\; \mbox{whenever }  k \mbox{ is odd},\\
\phi(n)-p_1+\frac{n}{2},\;\; \mbox{whenever }    k \mbox{ is even.}
                    \end{cases}
\end{equation}

We now construct a polynomial
$\Phi_n^*(x)\in \A_n$ from $\Phi^T_n(x)$  as follows. Let the degree of $\Phi_n^T(x)$ be $D$. Consider the monomials
in $\Phi_n^T(x)$ having exponent strictly between $D-n/2$ and $n/2$. If $x^b$ is the monomial
with smallest exponent among these, then $\Phi_n^*(x)=x^{b+n/2}+x^b-\Phi_n^T(x)$.
Since $n$ is even, $\Phi_n(x)$ divides $x^{n/2} + 1$ and hence $\Phi_n(x)^* \in \A_n$. Also, the monomial $x^b$ comes from the polynomial $f_1(x)$ and therefore
\begin{equation}\label{eqn:Phi}
\deg(\Phi_n^*(x))= \begin{cases}
\frac{n}{2}+\varphi(n), & {\mbox{ if }} k  {\mbox{ is odd,}} \\
\frac{n}{2}+\varphi(n)-1, & {\mbox{ if }} k  {\mbox{ is even.}}
\end{cases}
\end{equation}

 Since $\Phi_n^*(x)\in \A_n$, using Equation~(\ref{eqn:Phi}), the following result follows and hence the proof is omitted.

\begin{lemma}\label{lem:flat}
Let $n= 2 p_1 p_2\cdots p_k$ be the factorization of  $n$ into
odd primes  $p_1<p_2<\cdots <p_k$. Suppose that the cyclotomic polynomial $\Phi_n(x)$ is flat. Then
$$\min\{ \deg(f(x)) : f(x) \in \A_{n} \}\leq
\begin{cases}
\frac{n}{2}+\varphi(n), & {\mbox{\; if }}  k {\mbox{ is odd}}, \\
\frac{n}{2}+\varphi(n)-1, & {\mbox{ if }} k {\mbox{ is even}}.\\
\end{cases}$$
\end{lemma}

\begin{rem}
In general, we are not able to give exact comparison between the bounds obtained in Theorem~\ref{thm:main} and the bound in Lemma~\ref{lem:flat}. But it can be checked that whenever $3 p_1 < p_2$ then the bound in Lemma~\ref{lem:flat} is better than the bound in Theorem~\ref{thm:main}.
\end{rem}

\section*{Conclusion}

In this paper, we have tried to study the representations of subfields of a cyclotomic field with the help of circulant  and $0,1$-companion matrices. In particular, the following results have been obtained.

\begin{enumerate}
\item  A subfield of a cyclotomic field is representable by some circulant
matrix and conversely every circulant matrix represents a
subfield of a cyclotomic field.

\item Every real subfield of $\Q[\zeta_n]$ is representable by
a polynomial in the adjacency matrix of $C_n$, the cyclic graph.
Consequently, every real subfield of $\Q[\zeta_n]$ has integer symmetric
circulant matrix representation.

\item  Let $p$ be a prime and let $\K$ be a subfield $\Q[\ze_p]$. Then a $0,1$
circulant matrix $A$ of order $p$ is obtained such that $(A,\J)$ represents
$\K$.

\item Let $n= p^k$ for some prime $p$. Then the smallest $0,1$-companion matrix having $\ze_n$ as an eigenvalue  is $W_n$, the companion matrix of $x^n - 1$.

\item Let $n=p_1^{a_1} p_2^{a_2}$ be the prime factorization of $n$ as product of distinct primes. Then $\min\{ \deg(f(x)) : f(x) \in \A_{n} \} = \frac{n}{p_1p_2}(p_1p_2-1)$.

\item Let $n$ be a positive integer having $3$ or
more prime factors. Then  $\min\{ \deg(f(x)) : f(x) \in \A_{n} \} \le \frac{n}{p_1p_2}(p_1p_2-1)$, where $p_1$ and $p_2$ are the smallest two distinct primes dividing $n$. Furthermore, if  $n$ is even then this  upper bound is improved in Theorem~\ref{thm:main} and Lemma~\ref{lem:flat}.

It will be nice to improve the bounds obtained in this paper. Also, it will be nice to get examples where the bounds are attained.
\end{enumerate}


\begin{thebibliography}{20}
\bibitem{T:A}
Tom M. Apostol, {\em Introduction to Analytic Number theory},
Springer-Verlag, New York, (1976).

\bibitem{gb} Gennady Bachman, {\em Flat cyclotomic polynomials of order three},
Bull.London Math.Soc. 38  53-60 (2006).


\bibitem{bapat} R. B. Bapat, {\em Graphs and Matrices}, Springer, (2010).

\bibitem{biggs} N. L. Biggs, {\em Algebraic Graph Theory} (second edition),
Cambridge University Press, Cambridge, (1993).



\bibitem{davis} Philip J. Davis, {\em Circulant matrices}, A Wiley-Interscience publications, (1979).


\bibitem{D:F} David S. Dummit and Richard M. Foote, {\em Abstract Algebra} (second edition), John Wiley and Sons, (2002).

\bibitem{M:A} Michael Filaseta and Andrzej Schinzel, {\em On Testing the Divisibility of Lacunary Polynomials by Cyclotomic Polynomials}, Mathematics of Computation, Vol. 73, No. 246, pp. 957-965 (2004).

\bibitem{hof} A. J. Hoffman {\em On the polynomial of a graph}, The American
Mathematical Monthly, Vol. 70, No. 1 , pp. 30-36 (1963).


\bibitem{hoff}A. J. Hoffman and M. H. McAndrew, {\em The Polynomial of a Directed
Graph}, Proceedings of the American Mathematical Society, Vol. 16, No. 2, 303-309 (1965).


\bibitem{H:K} Kenneth Hoffman and Ray Kunge,{\em Linear Algebra} (second edition), Prentice-Hall, (1971).

\bibitem{kap1} Nathan Kaplan,  {\em Flat cyclotomic polynomials of order three},
Journal of Number Theory, 127, 118-126 (2007).

\bibitem{kap2} Nathan Kaplan, {\em  Flat cyclotomic polynomials of order four
and higher}, Integers, 10 , 357-363 (2010).



\bibitem{lehmer} Emma Lehmer, {\em On the magnitude of the coefficients of the
cyclotomic polynomial},  Bull. Amer. Math. Soc., Vol. 42,pp:389-392 (1936).

\bibitem{mor} Pieter Moree and Huib Hommerson, {\em Value distribution of Ramanujan sums and of  cyclotomic polynomial coefficients}, arXiv:math/0307352v1 [math.NT] 27 Jul (2003).


\bibitem{V:P} Victor V. Prasolov, {\em Polynomials}, Springer, (2001).



\bibitem{J:S} John P. Steinberger, {\em Minimal Vanishing Sums of Roots of Unity with Large Coefficients}, Proc. London Math. Soc., (3) 97, 689-717 (2008).

\bibitem{tan} R. Thangadurai, {\em On the coefficients of cyclotomic polynomials}, Proceedings of the Summer school on Cyclotomic Fields, June, 1999, Bhaskaracharya Pratishthana, Pune, pages 311-322, MR1802391 (2001k:11213) (2000).

\bibitem{turner}
James Turner, {\em Point-symmetric graphs with a prime number of points}, Journal of Combinatorial theory, Vol. 3, 136-145 (1967).


\end{thebibliography}
\end{document}